\documentclass{article}
\usepackage[english]{babel}
\usepackage[T1]{fontenc}
\usepackage[utf8]{inputenc}
\RequirePackage{amsthm,amsmath,amsfonts,amssymb}
\RequirePackage[numbers]{natbib}
\usepackage{forest}
\usetikzlibrary{calc}

\newcommand{\ac}[1]{\textcolor{red}{add citation}}

\usepackage{enumerate}
\usepackage{color}
\usepackage{multirow}\usepackage{graphicx}
\usepackage{amsfonts}
\usepackage{amsmath}
\usepackage{amssymb}

\usepackage{color}
\usepackage{dsfont}

\usepackage{pgfplots}
\usetikzlibrary{patterns}
\usepackage{amssymb}
\usepackage{graphicx}
\usepackage{amsmath}
\usepackage{bm}
\usepackage{ltablex}
\usepackage{algorithm} 
\usepackage{algorithmic}  
\usepackage[algo2e]{algorithm2e} 
\usepackage{caption}

\usetikzlibrary{decorations.pathreplacing}

\usetikzlibrary{calc,patterns}

\usepackage{url}
\usepackage{fancyhdr}
\usepackage{indentfirst}

\usepackage{enumerate}
\usepackage{dsfont}
\usepackage{tabulary}
\usepackage{booktabs}
\usepackage[colorlinks=true,citecolor=blue]{hyperref}
\usepackage{amsthm}
\usepackage{color}
\usepackage{comment}
\usepackage{tikz-cd}
\usepackage{graphicx}
\usepackage{amsfonts}
\usepackage{longtable}
\usepackage{amsmath}
\usepackage{amssymb}
\usepackage{url}
\usepackage{bbm}
\usepackage{fancyhdr}
\usepackage{indentfirst}
\usepackage{enumerate}
\usepackage{subcaption}
\usepackage{dsfont}
\pgfplotsset{compat=1.7}
\usepackage[colorlinks=true,citecolor=blue]{hyperref}
\usepackage{cleveref}
\usepackage{amsthm}
\usepackage{color}
\renewcommand{\cite}{\citet}
\usepackage{comment}
\usepackage[margin=1.2in]{geometry}
\usepackage{xcolor}

\usepackage{algorithm}

\renewcommand{\d}{\,\mathrm{d}}
\newcommand{\dd}{\overset{\d}{=}}
\newcommand{\p}{\mathbb{P}}
\newcommand{\q}{\mathbb{Q}}

\usepackage{xcolor}

\newcommand{\E}{\mathbb{E}}    
\newcommand{\R}{\mathbb{R}}    
\newcommand{\N}{\mathbb{N}}    

\theoremstyle{plain}
\newtheorem{theorem}{Theorem}
\newtheorem{corollary}[theorem]{Corollary}
\newtheorem{lemma}[theorem]{Lemma}
\newtheorem{proposition}[theorem]{Proposition}

\theoremstyle{definition}

\theoremstyle{remark}
\newtheorem{remark}{Remark}

\usepackage{tikz}

\usetikzlibrary{arrows.meta,positioning}

\newcommand{\ee}{\varepsilon}

\newcommand{\n}[1]{\left\lVert#1\right\rVert}
\newcommand{\A}{\mathcal{A}}
\newcommand{\cE}{\mathcal{E}}

\newcommand{\T}{\mathcal{T}}

\def\d{\mathrm{d}}

\def\lawis{\buildrel \mathrm{law} \over \sim}

\newcommand{\bone}{ {\mathbbm{1}} }

\usepackage{mathrsfs}
\newcommand{\M}{\mathcal{M}}

\newcommand{\C}{\mathcal{C}}

\newcommand{\lst}{\preceq_{\mathrm{st}}}

\renewcommand{\ge}{\geqslant}
\renewcommand{\le}{\leqslant}
\renewcommand{\geq}{\geqslant}
\renewcommand{\leq}{\leqslant}
\renewcommand{\epsilon}{\varepsilon}

\title{Boundedness of discounted branching random walks via generic chaining}

 \author{Zhenyuan Zhang\thanks{Department of Mathematics, Stanford University. Email: zzy@stanford.edu}}

\begin{document}
\maketitle

\begin{abstract}
    Consider a discrete-time supercritical discounted branching random walk, in which increments at depth $k$ are independent and identically distributed with the same law as $m^{-kH}Y$, where $Y$ has a fixed law, $H>0$, and $m>1$ is the expected number of offspring at depth one. We provide a clean characterization of the boundedness of the discounted branching random walk: under mild conditions on the offspring distribution, the process is almost surely bounded if and only if $\mathbb{E}[|Y|^{1/H}]<\infty$. This extends results of Athreya (1985) and A{\"i}d{\'e}kon--Hu--Shi (2024), and provides a partial answer to Open Problem 31 of Aldous--Bandyopadhyay (2005).
\end{abstract}


\section{Introduction}

Consider a supercritical Galton--Watson tree and let $\T$ be the set of rays (paths) from the root to the boundary of the tree. Let $V_n$ denote the set of vertices at depth $n$ (i.e., of distance $n$ from the root) and for $v\in V_n$, write $\ell(v)=n$. Denote by $V=\bigcup_{n\in\N}V_n$ the set of all (non-root) vertices. For each vertex $v\in V$ we associate an i.i.d.~real-valued random variable $\eta_v$ with the same distribution as a fixed random variable $Y$. 
This leads to a \textit{branching random walk} on $\R$. For each ray $t\in\T$ and $i\in \N$, we define $t_i$ to be the unique vertex in $V_i$ on $t$, i.e., $i=\ell(t_i)$. Then, the set $\{\sum_{i=1}^N\eta_{t_i}:t\in\T\}$ describes the set of locations of particles at depth $N$ of a branching random walk.  
Classical results state that the maximum at depth $N$ typically grows linearly in $N$ with a logarithmic correction and is tight around its median \citep{addario2009minima,aidekon2013convergence}.

For a finite set $A$, we let $\#A$ denote its cardinality. Denote by $Z=\#V_1$ the offspring distribution. Let $H>0$ and $m:=\E[Z]>1$. We define the \textit{discounted branching random walk} as
\begin{align}
    X_t:=\sum_{i\geq 1} m^{-iH}\eta_{t_i},\quad t\in\T.\label{eq:discounted BRW}
\end{align}
For simplicity, we assume that $\E[|Y|^\alpha]<\infty$ for some $\alpha>0$, so that \eqref{eq:discounted BRW} is well-defined. 
We say $\{X_t\}_{t\in\T}$ is \textit{bounded} if $\sup_{t\in\T}|X_t|<\infty$ a.s.\footnote{In the extreme case where $\T$ is empty, we define $\sup_{t\in\T}X_t=\sup_{t\in\T}|X_t|=0$.} 

If $\p(|Y|>x)\sim x^{-1/\theta}$ for some $\theta>0$, previous studies \citep{aidekon2024boundedness,athreya1985discounted} indicate a phase transition of boundedness at $H=\theta$: $\{X_t\}_{t\in\T}$ is bounded if $H>\theta$ and unbounded if $H<\theta$; see further discussions below. Our main result characterizes the boundedness in the critical case, without imposing any regularity assumption on the tail of $Y$ (such as regular variation).

\begin{theorem}
    \label{thm:final}
    Assume that $m>1$ and $H>0$. The following statements hold.
\begin{enumerate}[(i)]
    \item If there exists $q>\max\{1/H,1\}$ such that $\E[Z^q]<\infty$ and $\E[|Y|^{1/H}]<\infty$, then $\sup_{t\in\T}|X_t|<\infty$ a.s.
    \item If $\p(Z=0)=0$, $\E[Z\log Z]<\infty$, and $\E[|Y|^{1/H}]=\infty$, then $\sup_{t\in\T}|X_t|=\infty$ a.s.
\end{enumerate} In particular, if $Z$ has finite moments of all orders and $\p(Z=0)=0$, then  for all $H>0$,  
$\sup_{t\in\T}|X_t|<\infty$ a.s.~if and only if $\E[|Y|^{1/H}]<\infty$.
\end{theorem}

\begin{remark}
Let us briefly explain why $\E[|Y|^{1/H}]<\infty$ is the critical threshold, assuming for simplicity that $Z=m$ almost surely. Intuitively, the threshold arises from the competition of two forces: the discounting factor $m^{-kH}$ and the large values of the increments due to the heavy tail of $Y$. An elementary argument shows that 
$$\E[|Y|^{1/H}]=\infty\quad \Longleftrightarrow \quad\sum_{k\geq 1} m^k\,\p(|Y|>m^{kH})=\infty;$$
see Lemma \ref{lem:polyLambda} below. 
By the Borel--Cantelli lemma, $\sum_{k\geq 1} m^k\,\p(|Y|>m^{kH})=\infty$ if and only if there exist almost surely infinitely many vertices $v$ with $|m^{-\ell(v)H}\eta_v|> 1$, and the same holds if we consider any fixed subtree of $\T$. In this case, one can recursively find infinitely many large values in $\T$ that lie on the same ray, leading to unboundedness in view of \eqref{eq:discounted BRW}. On the other hand, if $\E[|Y|^{1/H}]<\infty$, Lemma \ref{lem:polyLambda} shows that $\sum_{k\geq 1} m^k\,\p(|Y|>u\,m^{kH})<\infty$ for each fixed $u>0$, and hence it becomes difficult to find a ray with large values.
\end{remark}

To prove Theorem \ref{thm:final}, the essential part is to understand the case where $Y$ is symmetric, where we exploit the connection to chaining on Bernoulli processes via a Rademacher symmetrization. The general asymmetric case requires further work, but is relatively standard. Note that if we assume that $Y$ has a symmetric distribution, we have $\{X_t\}_{t\in\T}\dd \{-X_t\}_{t\in\T}$, so our notion of boundedness $\sup_{t\in\T}|X_t|<\infty$ a.s.~is equivalent to $\sup_{t\in\T}X_t<\infty$ a.s.

If $Y$ is symmetric, we apply generic chaining techniques to prove the boundedness of $\{X_t\}_{t\in\T}$. Roughly speaking, chaining techniques yield upper and lower bounds for $\E[\sup_{t\in T}X_t]$ for a generic stochastic process $\{X_t\}_{t\in T}$ \citep{talagrand2022upper}. Nevertheless, since in our setting $Y$ may not have a finite first moment, the expected supremum $\E[\sup_{t\in \T}X_t]$ might still be infinite despite that $\sup_{t\in \T}X_t$ may be finite a.s. To address this issue, we apply a symmetrization procedure and view it as a supremum of a random Bernoulli process, which is of the form $\sup_{t\in S}\sum_{i\geq 1}t_i\ee_i$, where $\{\ee_i\}_{i\geq 1}$ are i.i.d.~Rademacher random variables (for some random subset $S\subseteq\ell^2(\N)$ constructed from $\T$). Then, we employ classical bounds on Bernoulli processes (Lemma \ref{lemma:Bernoulli} below) and show that the expected supremum $\E[\sup_{t\in S}\sum_{i\geq 1}t_i\ee_i]$ is a.s.~finite. As a by-product of our approach, we have the following result on the finiteness of the expected supremum.\footnote{Note that our technique also yields an \textit{explicit} upper bound of $\E[\sup_{t\in\T}|X_t|]$ in terms of the increment law $Y$, the offspring distribution $Z$, and the discounting factor $H$. However, we do not pursue it in this work as it is likely not tight.}

\begin{proposition}\label{prop:finite E 2}
 \sloppy   Assume the same conditions as in Theorem \ref{thm:final}(i). If $H\in(0,1)$, we have $\E[\sup_{t\in\T}|X_t|]<\infty$.
\end{proposition}

Our approach is partly inspired by the following connection: the random Fourier series
$$X(\theta):=\sum_{n=1}^\infty \frac{1}{\sqrt{n}}(A_n\cos(n\theta)+B_n\sin(n\theta)),\quad \theta\in[0,2\pi),$$
where $\{A_n\}_{n\geq 1}$ and $\{B_n\}_{n\geq 1}$ are i.i.d.~standard Gaussian, 
is well-known as a canonical log-correlated field \citep{huang2025random}: $\E[X(\theta_1)X(\theta_2)]\approx -\log|e^{i\theta_1}-e^{i\theta_2}|$. Branching random walks (without discounting) also constitute simple examples in the log-correlated field class. The general problem of characterizing the boundedness of random Fourier series 
$$\sum_{k\geq 1}\xi_k e^{ik\theta},\quad \theta\in[0,2\pi),$$
where $\{\xi_k\}_{k\geq 1}$ are independent symmetric random variables, has been fully resolved in Chapter 7 of \citep{talagrand2022upper} using chaining with Bernoulli processes; see also references therein for a historical account. Therefore, it is reasonable that similar (and even simpler) chaining techniques may be transferred to understand the boundedness of discounted branching random walks.

In what follows, we discuss a few motivations for discounted branching random walks and related works, as well as several immediate consequences of our main results.
\begin{itemize}
    \item Suppose that we are in the special case where $\p(Z=m)=1$ and $Y\geq 0$. The distribution of the maximum of a discounted branching random walk of the form \eqref{eq:discounted BRW} solves the recursive distributional equation 
    \begin{align}
        X\dd Y+\max_{1\leq i\leq m}cX_i,\quad c\in(0,1),\label{eq:RDE}
    \end{align}
    where $\{X_i\}_{1\leq i\leq m}$ are i.i.d.~copies of $X$ and independent of $Y$.  The existence of a positive solution to \eqref{eq:RDE} is equivalent to the boundedness of a discounted branching random walk with discounting factor $c$ and deterministic branching into $m$ offspring at each step. This connection is a particular case of the discussion in Section 4.4 of the survey \citep{aldous2005survey}. In particular, Open Problem 31 therein asks the question of establishing conditions for a.s.~boundedness of discounted branching random walks and their generalizations. We refer to \citep{aldous2005survey} for equations related to \eqref{eq:RDE} and to \citep{neininger2005analysis,ruschendorf2006stochastic} for a general fixed-point theory for recursions that leads to broad sufficient conditions for the existence of the solutions. 

In a similar vein, \citep{athreya1985discounted} shows that the distribution function $G$ of $\sup_{t\in\T}X_t$ satisfies the functional integral equation 
\begin{align}
    G(x)=\int_0^xf\Big(G\Big(\frac{x-y}{c}\Big)\Big)\,\d F(y),\label{eq:FIE}
\end{align}
where $f(s)=\E[s^Z]$ and $F$ is the distribution function of the increment $Y$. 
Assuming the existence of $\theta>0$ such that $\sup_{x>0}x^\theta\p(Y>x)<\infty$ and $mc^\theta<1$, Proposition 2 of \citep{athreya1985discounted} showed that the discounted branching random walk is a.s.~bounded and hence \eqref{eq:FIE} has a solution.\footnote{Suppose that there exists $\theta>0$ with $\sup_{x>0}x^\theta\p(Y>x)<\infty$ and $mc^\theta<1$. Since $c=m^{-H}$, we have $\theta>1/H$, and hence $\E[|Y|^{1/H}]<\infty$. Therefore, Theorem \ref{thm:final} is stronger than Proposition 2 of \citep{athreya1985discounted}. See also Remark 4 therein for a slightly weaker condition than that in Proposition 2. } 
The a.s.~boundedness has been recently extended by \citep{aidekon2024boundedness}, focusing on branching random walks with general (and even random) discounting, partly motivated by self-similar Markov trees \citep{bertoin2024self}. Their main results established an explicit phase transition threshold of the boundedness in terms of the tail exponent of the increment distribution (which is assumed to exist). To the best of our knowledge, the critical phase remains unsolved.

For a geometrically discounted branching random walk, our results complete the picture by fully characterizing the boundedness in the critical case, under only minor assumptions on the offspring distribution. Compared to results in \citep{aidekon2024boundedness,athreya1985discounted}, our setting is more general: we do not assume $Y\geq 0$. Note that the techniques applied in this work (generic chaining) are different from those in \citep{aidekon2024boundedness} (many-to-one lemma and local time analysis) and \citep{athreya1985discounted} (elementary Borel--Cantelli arguments).  The following result follows directly from Theorem \ref{thm:final}.

\begin{corollary}
Assume that $m>1$.  Suppose that the increment $Y$ satisfies $Y\geq 0$ and $\E[Y^{1/H}]<\infty$. Then for $c=m^{-H}$, \eqref{eq:RDE} admits a positive solution.
\end{corollary}

    \item Discounted branching random walks are related to a $p$-adic analogue of self-similar processes with stationary increments (sssi processes). In the typical setting where the index set is $\R$, sssi processes include many well-known examples such as fractional Brownian motion and stable L\'{e}vy processes \citep{embrechts2002selfsimilar}, and often possess neat properties due to probabilistic symmetry. On the other hand, sssi processes indexed by the $p$-adic numbers $\q_p$ arise naturally when considering the discrete skeleton of the notion of sssi processes \citep{shen2021discrete}, and they typically exhibit completely different and obscure behavior. 
    
  Let $p$ be a prime.  An interesting class of examples of $p$-adic sssi processes is provided in Example 4.1 of \citep{shen2021discrete}, where the discrete skeleton $\{X_n\}_{n\geq 0}$ (using the natural embedding $\N\cup\{0\}\hookrightarrow \q_p$) is defined as follows: let $\{Y^k_n\}_{k\in\N,0\leq n\leq p^k-1}$ be a triangular array of i.i.d.~random variables with the same law as $Y$, and extend it periodically to $\{Y^k_n\}_{k\in\N,n\geq 0}$ by defining $Y^k_\ell=Y^k_n$ for $\ell\equiv n$ (mod $p^k$). Let $c\in(0,1)$ and define
    \begin{align}
        X_n:=\sum_{k=1}^\infty c^k (Y^k_n-Y^k_0),\quad n\geq 0.\label{eq:sssi}
    \end{align}
    Note that this is a re-centered discounted branching random walk with deterministic branching into $p$ offspring at each step. Theorem \ref{thm:final} gives a complete characterization of the boundedness of this class of $p$-adic sssi processes. In general, sample path properties of sssi processes on the real line are well-studied \citep{samorodnitsky2017stable,vervaat1985sample}, but results on $\q_p$ are scarce. The following result follows directly from Theorem \ref{thm:final}.

\begin{corollary}
    The process $\{X_n\}_{n\geq 0}$ in \eqref{eq:sssi} has bounded sample paths if and only if $\E[|Y|^{-\frac{\log p}{\log c}}]<\infty$.
\end{corollary}

    \item In addition, discounted branching random walks are also connected to various settings where genealogy induces a geometric scaling by generation, such as randomly-growing binary trees \citep{aldous1988diffusion}, delayed Yule processes \citep{dascaliuc2018delayed,dascaliuc2021doubly}, self-similar fragmentation \citep{dyszewski2022sharp}, and self-similar measures \citep{benjamini2009branching}.
\end{itemize}

\textbf{Outline}. In the rest of the paper, we prove our main results: Theorem \ref{thm:final} and Proposition \ref{prop:finite E 2}. We start from the case where $Y$ is symmetric, and establish upper bounds for the (expected) supremum in Section \ref{sec:suff}. The general asymmetric case will be presented in Section \ref{sec:general}.



\section{Sufficient condition for boundedness in the symmetric case}\label{sec:suff}

The goal of this section is to establish the sufficiency part of boundedness (Theorem \ref{thm:final}(i)) in the case where the distribution of $Y$ is symmetric. 

\begin{theorem}\label{thm}
    Assume that $m>1$, $H>0$, and the increment distribution $Y$ is symmetric. If there exists $q>\max\{1/H,1\}$ such that $\E[Z^q]<\infty$ and $\E[|Y|^{1/H}]<\infty$, then $\sup_{t\in\T}X_t<\infty$ a.s.
\end{theorem}

We denote by
\begin{align}
    P:=\sum_{k\geq 1} m^k\,\p(|Y|>m^{kH}).\label{eq:P}
\end{align}
We may assume that $P>0$, since otherwise, if $P=0$, $|Y|$ is bounded, which directly implies the boundedness of the discounted branching random walk by a deterministic bound. We also define 
\begin{align}
    u_n:=2^{-H2^{n}+Hn},\quad n\geq 1.\label{eq:undef}
\end{align}

\subsection{Elementary lemmas}

In this section, we provide a few elementary lemmas that will be of future use. Let $q$ be a number such that $q>\max\{1/H,1\}$ and $\E[Z^q]<\infty$, where we recall that $Z$ is the offspring distribution. 

\begin{lemma}\label{lem:polyLambda}
We have $\E[|Y|^{1/H}]<\infty$ if and only if $P<\infty$. In this case, for all $u\in(0,1]$,
\begin{equation}\label{eq:polyLambda}
\sum_{k\geq 1} m^k\,\p(|Y|>u\,m^{kH})\le mu^{-1/H}\Big(P+\frac{1}{m-1}\Big)=:C_{m,P}\,u^{-1/H}.
\end{equation}
Otherwise, for all $u>0$,
\begin{align}
    \sum_{k\geq 1} m^k\,\p(|Y|>u\,m^{kH})=\infty.\label{eq:2}
\end{align}
\end{lemma}

\begin{proof}
To prove the first equivalence, observe that
\begin{align}
    \E[|Y|^{1/H}]&=\int_0^\infty \p(|Y|\geq x^H)\,\d x=O(1)+\sum_{k\geq 1} \int_{m^k}^{m^{k+1}}\p(|Y|\geq x^H)\,\d x.\label{eq:1}
\end{align}
On the other hand, for $k\geq 1$,
\begin{align}
 \frac{m-1}{m}\,m^{k+1}\p(|Y|>m^{(k+1)H})   \leq \int_{m^k}^{m^{k+1}}\p(|Y|\geq x^H)\,\d x\leq m(m-1)m^{k-1}\p(|Y|>m^{(k-1)H}),\label{eq:2sides}
\end{align}where the index shift on the right-hand side deals with possible atoms of $|Y|$ at $m^{kH}$. 
Summing \eqref{eq:2sides} over $k$ and using \eqref{eq:1} shows that $\E[|Y|^{1/H}]<\infty$ if and only if $P<\infty$.

To prove the second claim, we set  $r:=\big\lceil \frac{\log(1/u)}{H\log m}\big\rceil$, so that
$$
\p(|Y|>u m^{kH})\le \p(|Y|>m^{(k-r)H}),
$$
and hence
\begin{align*}
    \sum_{k\geq 1}m^k\p(|Y|>u m^{kH})
&\le m^r\sum_{\ell\ge 1-r} m^\ell \p(|Y|>m^{\ell H})\\
&\leq m^r\Big(P+\sum_{\ell=1-r}^{-1} m^\ell\Big)\leq mu^{-1/H}\Big(P+\frac{1}{m-1}\Big).
\end{align*} This yields \eqref{eq:polyLambda}. The claim \eqref{eq:2} can be verified similarly using \eqref{eq:1}.
\end{proof}

\begin{lemma}\label{lemma: un}
Assume that $P<\infty$.    Almost surely for $n$ large enough, $\#\{v\in V:|m^{-H\ell(v)}\eta_v|> u_n\}\leq 2^{2^n}-1$. 
\end{lemma}
\begin{proof}
    By Fubini's theorem, \eqref{eq:undef}, and Lemma \ref{lem:polyLambda},
    \begin{align}
        \E\big[\#\big\{v\in V:|m^{-H\ell(v)}\eta_v|> u_n\big\}\big]&=\sum_{k\geq 1}m^k\p\big(|Y|>u_nm^{kH}\big)\leq C_{m,P}\, 2^{2^n-n}.\label{eq:M}
    \end{align}
    The rest follows from Markov's inequality and the Borel--Cantelli lemma.
\end{proof}

 For a vertex $v\in V$, we write $v\in t$ if $t$ goes through $v$, i.e., there is $i\in\N$ such that $t_i=v$. Let $\bone_E$ denote the indicator random variable of an event $E$.

\begin{lemma}\label{lemma:tree exceedance}
Assume that $P<\infty$.  There exists a constant $C>0$ depending only on $q$ and the law of $Z$ such that for any $\alpha>1$ and $r\geq \alpha P$,
    $$\p\Big(\max_{t\in\T}\#\{w\in t:\,|m^{-H\ell(w)}\eta_w|>1\}\geq r\Big)\leq \exp\Big(-r\log\big(\frac{r}{\alpha P}\big)+r-\alpha P\Big)+C\alpha^{-q}.$$
\end{lemma}

\begin{proof}
Denote by $W_k:=\#V_k/m^k$. By the Kesten--Stigum theorem (see Theorem I.10.1 of \citep{athreya2012branching}) and using $\E[Z^q]<\infty\implies \E[Z\log Z]<\infty$, we have $W_k\to W$ a.s.~for some nonnegative random variable $W$. On the other hand, Theorem 5 of \citep{bingham1974asymptotic} and our assumption $\E[Z^q]<\infty$ yield that $\E[W^q]<\infty$. Theorem 1.2 of \citep{alsmeyer2004existence} then gives $\E[(\sup_{k\geq 1}W_k)^q]<\infty$. Let $C=\E[(\sup_{k\geq 1}W_k)^q]$, so by Markov's inequality, 
\begin{align}
    \p(\exists k\geq 1,~\#V_k> \alpha m^k)=\p\Big(\sup_{k\geq 1}W_k>\alpha\Big)\leq C\alpha^{-q},\quad \alpha>1.\label{eq:3}
\end{align}

Next, we restrict to the event that for any $ k\geq 1,~\#V_k\leq  \alpha m^k$.    We bound $\max_{t\in\T}\#\{w\in t:\,|m^{-H\ell(w)}\eta_w|>1\}$ by the total number $N$ of vertices $w\in V$ with $|m^{-H\ell(w)}\eta_w|>1$. Conditioned on the Galton--Watson tree, the contribution from depth $k$ satisfies $N_k\lawis\mathrm{Bin}(\#V_k,\p(|Y|>m^{kH}))$, and hence is stochastically dominated by $\mathrm{Bin}(\alpha m^k,\p(|Y|>m^{kH}))$ on the event $\{\forall k\geq 1,~\#V_k\leq  \alpha m^k\}$.\footnote{We assume without loss of generality that $\alpha m^k\in\N$ as it does not affect the asymptotic results.} Therefore, with $N:=\sum_{k\geq 1}N_k$, for each $\lambda>0$,
    \begin{align*}
        \E[e^{\lambda N}\bone_{\{\forall k\geq 1,~\#V_k\leq  \alpha m^k\}}]&\leq\prod_{k\geq 1}(1-\p(|Y|>m^{kH})+\p(|Y|>m^{kH})e^\lambda)^{\alpha m^k}\\
        &\leq \exp\Big((e^\lambda-1)\sum_{k\geq 1}\alpha m^k\p(|Y|>m^{kH})\Big)\leq e^{(e^\lambda-1)\alpha P}.
    \end{align*}
    Optimizing over $\lambda>0$ and using Markov's inequality yields that
$$\p\Big(\max_{t\in\T}\#\{w\in t:\,|m^{-H\ell(w)}\eta_w|>1\}\geq r; ~\forall k\geq 1,~\#V_k\leq  \alpha m^k\Big)\leq \exp\Big(-r\log\big(\frac{r}{\alpha P}\big)+r-\alpha P\Big).$$    
 This along with \eqref{eq:3} gives the desired result.
\end{proof}

\subsection{Preliminaries and setup}

For a subset $S$ of $\ell^2=\ell^2(\N)$, we define the Bernoulli process
$$B_t:=\sum_{i\geq 1}t_i\ee_i,\quad t\in S,$$
where $\{\ee_i\}_{i\in\N}$ is a sequence of i.i.d.~Rademacher random variables (i.e., uniformly distributed on $\{\pm 1\}$). Set 
$$b(S):=\E\Big[\sup_{t\in S}B_t\Big].$$
Furthermore, define the $\gamma_2$ functional as follows: for $S\subseteq\ell^2$, let
$$\gamma_2(S):=\inf\sup_{t\in S}\sum_{n\geq 0}2^{n/2}\Delta(A_n(t)),$$
where $A_n(t)$ is the unique element $A\in\A_n$ that contains $t$, $\Delta(A_n(t))$ denotes the diameter of $A_n(t)$ for the $\ell^2$ metric, and the infimum is taken over all \textit{admissible sequences of partitions} $\{\A_n\}_{n\geq 0}$, which by definition means that the sequence of partitions $\{\A_n\}_{n\geq 0}$ is increasing, $\#\A_0=1$, and $\#\A_n\leq 2^{2^n}$ for each $n\in\N$. The following result is Proposition 6.2.7 of \citep{talagrand2022upper}.

\begin{lemma}\label{lemma:Bernoulli}
    There exists a universal constant $L>0$ such that for any $S\subseteq\ell^2$ and choices of $S_1,S_2$ such that $S\subseteq S_1+S_2$, we have
    \begin{align}
        b(S)\leq L\Big(\gamma_2(S_1)+\sup_{t\in S_2}\n{t}_1\Big).\label{eq:b}
    \end{align}
\end{lemma}

\begin{remark}
   The inequality \eqref{eq:b} can be reversed, up to universal constants, if one takes an infimum on the right-hand side over all possible decompositions $S\subseteq S_1+S_2$, a result known as Talagrand's Bernoulli conjecture \citep{talagrand2005genericchaining} and later resolved in \citep{bednorz2014boundedness}. This result lies in the core of understanding boundedness of Bernoulli processes and random series of functions \citep{talagrand2022upper}. For the purpose of this work, the upper bound in Lemma \ref{lemma:Bernoulli} suffices.
\end{remark}

The first step is to perform symmetrization along each depth $i\geq 1$ by introducing Rademacher noise: recalling \eqref{eq:discounted BRW} and the symmetry of $Y$, we have
\begin{align}
    \{X_t\}_{t\in\T}\dd  \Big\{\sum_{i\geq 1}\ee_i m^{-iH}\eta_{t_i}\Big\}_{t\in\T},\label{eq:dd}
\end{align}
and hence
$$\sup_{t\in\T}X_t\dd\sup_{t\in\T}\sum_{i\geq 1}\ee_i m^{-iH}\eta_{t_i}.$$
We view the right-hand side of \eqref{eq:dd} as a random Bernoulli process, where the randomness comes from $\eta_{t_i}$. Conditioned on the random variables $\{\eta_v\}_{v\in V}$, we identify $\T$ as a subset $S=S(\T)$ of $\ell^2$, where
$$t\in \T\longleftrightarrow s=s(t)=(m^{-iH}\eta_{t_i})_{i\in\N}\in\ell^2.$$
At this point, it is not entirely clear that $S\subseteq \ell^2$ a.s., but this will be established in \eqref{eq:finite diameter} below using arguments independent of chaining.

To prove the a.s.~boundedness of the discounted branching random walk, it remains to show that $b(S)<\infty$ a.s., and hence we need to find decompositions $S\subseteq S_1+S_2$ for each realization of the random set $S\subset\ell^2$ to apply Lemma \ref{lemma:Bernoulli}. This decomposition is simple: we simply extract coordinates with absolute value greater than one and put them into $S_2$, and the rest into $S_1$. Define 
\begin{align}
    s_1(t):=(m^{-iH}\eta_{t_i}\bone_{\{|m^{-iH}\eta_{t_i}|\leq 1\}})_{i\in\N}\quad\text{and}\quad s_2(t):=(m^{-iH}\eta_{t_i}\bone_{\{|m^{-iH}\eta_{t_i}|> 1\}})_{i\in\N}.\label{eq:s1s2}
\end{align}
Also let
$$S_1=S_1(\T):=\{s_1(t):t\in\T\}\quad\text{and}\quad S_2=S_2(\T):=\{s_2(t):t\in\T\}.$$
Then $S\subseteq S_1(\T)+S_2(\T)$ since $s(t)=s_1(t)+s_2(t)$ for each $t\in\T$. By the Borel--Cantelli lemma and the Kesten--Stigum theorem, if $P<\infty$, there are almost surely finitely many vertices $v$ such that $|m^{-H\ell(v)}\eta_v|>1$. It follows that 
\begin{align}
    \sup_{t\in\T}\n{s_2(t)}_1<\infty \quad\text{a.s.}\label{eq:s2}
\end{align}
Therefore, by Lemma \ref{lemma:Bernoulli}, we are left with proving 
\begin{align}
    \gamma_2(S_1(\T))<\infty\quad\text{a.s.}\label{eq:t}
\end{align}

\subsection{Constructing admissible sequences}\label{sec:An}

The goal of this section is to construct, for each realization of $S_1$, an admissible sequence of partitions $\{\A_n\}_{n\geq 0}$ of $S_1$ satisfying $\A_0=\{S_1\}$ and $\#\A_n\leq 2^{2^n}$. The partition $\A_n$ will be the product of two partitions $\A^{(1)}_n$ and $\A^{(2)}_n$ of $S_1$,\footnote{For two partitions $\A=\{A_1,\dots,A_m\}$ and $\mathcal B=\{B_1,\dots,B_\ell\}$ of the same set, we define their product as the partition $\{A_i\cap B_j:1\leq i\leq m,\,1\leq j\leq \ell\}$.} which we describe in detail below. For $v,w\in V$, we write $v\succeq w$ if $v$ is a genealogical descendant of $w$. 

First, we choose the largest integer $h=h(n)$ such that $m^h\leq 2^{2^{n-2}}$. By the Kesten--Stigum theorem and the Borel--Cantelli lemma, there exists an integer-valued random variable $N_1$ (chosen as small as possible) such that $\#V_k\leq k^2m^k$ a.s.~for $k\geq h(N_1)$. Let also $N_2\in\N$ be such that for $n\geq N_2$, $h(n)^2m^{h(n)}\leq 2^{2^{n-1}}$. For $n\geq 0$, define the partition $\A^{(1)}_n$ of $S_1$ as follows:
\begin{itemize}
    \item if $n\leq \max\{N_1,N_2\}$, let $\A^{(1)}_n=\{S_1\}$;
    \item  otherwise, the partition $\A^{(1)}_n$ is such that two sequences in $S_1$ belong to the same set in the partition if and only if their corresponding vertices have a common ancestor at depth $h$.
\end{itemize} 
Observe that by the construction above, $\#\A^{(1)}_n\leq \#V_h\leq h^2m^h\leq 2^{2^{n-1}}$. 

Second, recall from Lemma \ref{lemma: un} that there exists a $\sigma(S)$-measurable integer-valued random variable $N_3$ such that for $n\geq N_3$, $\#\{v\in V:|m^{-H\ell(v)}\eta_v|> u_n\}\leq 2^{2^n}-1$. For $n\geq 0$, define the partition $\A^{(2)}_n$ of $S_1$ as follows:
\begin{itemize}
    \item for $n\leq N_3$, let $\A^{(2)}_n=\{S_1\}$;
    \item otherwise, we consider the set of maximal elements $\M_{n-1}\subseteq V$ in $\{v\in V:|m^{-H\ell(v)}\eta_v|> u_{n-1}\}$ under the partial order $\succeq$, so $\#\M_{n-1}\leq 2^{2^{n-1}}-1$. For each $w\in \M_{n-1}$, we collect the set of rays $R_w\subseteq S_1$ that go through $w$ as an element in $\A^{(2)}_n$, and also put $S_1\setminus \bigcup_{w\in\M_{n-1}}R_w$ in $\A^{(2)}_n$.
\end{itemize}
 It follows that $\#\A^{(2)}_n\leq 2^{2^{n-1}}$, and hence the product satisfies $\#\A_n\leq \#\A^{(1)}_n\#\A^{(2)}_n\leq  2^{2^{n-1}}\times 2^{2^{n-1}}=2^{2^{n}}$. 

In summary, our construction yields that, for $n>\max\{N_1,N_2,N_3\}$, if $s,t\in S_1$ belong to the same set in the partition $\A_n$ (which we recall is the product of $\A^{(1)}_n$ and $\A^{(2)}_n$),
\begin{itemize}
    \item $s,t$ have a common ancestor at depth $h=\lfloor 2^{n-2}(\log_m2)\rfloor$;
    \item the coordinates where $s,t$ differ all belong to $[-u_{n-1},u_{n-1}]$. In particular, $\n{s-t}_\infty\leq 2u_{n-1}$.
\end{itemize}
See Figure \ref{fig:1} for an illustration. Our choices of $N_1,N_2,N_3$ are used to control the sizes of the partitions.  
Moreover, our construction implies that each set in $\A_n$ corresponds to a subtree of $S_1$. We collect the roots of the subtrees into a set $\C_n\subseteq V$. Denote by $\T_v$ the subtree with root $v$.

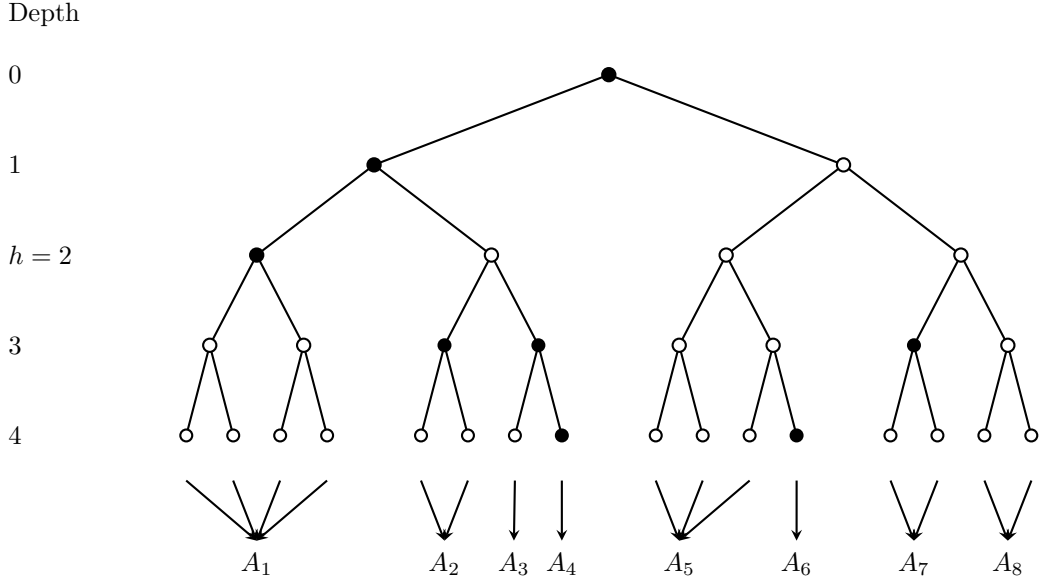
\begin{figure}[!ht]
    \centering
   \begin{tikzpicture}[line width=0.8pt, >=stealth,xscale=1.3]
  \usetikzlibrary{calc}
  \tikzset{
    open/.style   ={circle, draw, fill=white, inner sep=0pt, minimum size=5pt},
    filled/.style ={circle, draw, fill=black, inner sep=0pt, minimum size=5pt},
    leafopen/.style   ={circle, draw, fill=white, inner sep=0pt, minimum size=4.5pt},
    leaffilled/.style ={circle, draw, fill=black, inner sep=0pt, minimum size=4.5pt},
  }

  \coordinate (R0) at ( 0.00,  0.0);

  \coordinate (L1) at (-2.40, -1.2);
  \coordinate (R1) at ( 2.40, -1.2);

  \coordinate (LL2) at (-3.60, -2.4);
  \coordinate (LR2) at (-1.20, -2.4);
  \coordinate (RL2) at ( 1.20, -2.4);
  \coordinate (RR2) at ( 3.60, -2.4);

  \coordinate (LLL3) at (-4.08, -3.6);
  \coordinate (LLR3) at (-3.12, -3.6);

  \coordinate (LRL3) at (-1.68, -3.6);
  \coordinate (LRR3) at (-0.72, -3.6);

  \coordinate (RLL3) at ( 0.72, -3.6);
  \coordinate (RLR3) at ( 1.68, -3.6);

  \coordinate (RRL3) at ( 3.12, -3.6);
  \coordinate (RRR3) at ( 4.08, -3.6);

  \coordinate (LLLL4) at (-4.32, -4.8);
  \coordinate (LLLR4) at (-3.84, -4.8);

  \coordinate (LLRL4) at (-3.36, -4.8);
  \coordinate (LLRR4) at (-2.88, -4.8);

  \coordinate (LRLL4) at (-1.92, -4.8);
  \coordinate (LRLR4) at (-1.44, -4.8);

  \coordinate (LRRL4) at (-0.96, -4.8);
  \coordinate (LRRR4) at (-0.48, -4.8);

  \coordinate (RLLL4) at ( 0.48, -4.8);
  \coordinate (RLLR4) at ( 0.96, -4.8);

  \coordinate (RLRL4) at ( 1.44, -4.8);
  \coordinate (RLRR4) at ( 1.92, -4.8);

  \coordinate (RRLL4) at ( 2.88, -4.8);
  \coordinate (RRLR4) at ( 3.36, -4.8);

  \coordinate (RRRL4) at ( 3.84, -4.8);
  \coordinate (RRRR4) at ( 4.32, -4.8);

  \draw (R0) -- (L1);
  \draw (R0) -- (R1);

  \draw (L1) -- (LL2);
  \draw (L1) -- (LR2);

  \draw (R1) -- (RL2);
  \draw (R1) -- (RR2);

  \draw (LL2) -- (LLL3);
  \draw (LL2) -- (LLR3);

  \draw (LR2) -- (LRL3);
  \draw (LR2) -- (LRR3);

  \draw (RL2) -- (RLL3);
  \draw (RL2) -- (RLR3);

  \draw (RR2) -- (RRL3);
  \draw (RR2) -- (RRR3);

  \draw (LLL3) -- (LLLL4);
  \draw (LLL3) -- (LLLR4);

  \draw (LLR3) -- (LLRL4);
  \draw (LLR3) -- (LLRR4);

  \draw (LRL3) -- (LRLL4);
  \draw (LRL3) -- (LRLR4);

  \draw (LRR3) -- (LRRL4);
  \draw (LRR3) -- (LRRR4);

  \draw (RLL3) -- (RLLL4);
  \draw (RLL3) -- (RLLR4);

  \draw (RLR3) -- (RLRL4);
  \draw (RLR3) -- (RLRR4);

  \draw (RRL3) -- (RRLL4);
  \draw (RRL3) -- (RRLR4);

  \draw (RRR3) -- (RRRL4);
  \draw (RRR3) -- (RRRR4);

  \node[filled] at (R0)  {};
  \node[filled] at (L1)  {};
  \node[open]   at (R1)  {};

  \node[filled] at (LL2) {};
  \node[open]   at (LR2) {};
  \node[open]   at (RL2) {};
  \node[open]   at (RR2) {};

  \node[open] at (LLL3) {};
  \node[open] at (LLR3) {};

  \node[leaffilled] at (LRL3) {};
  \node[leaffilled] at (LRR3) {};

  \node[open] at (RLL3) {};
  \node[open] at (RLR3) {};

  \node[leaffilled] at (RRL3) {};
  \node[open] at (RRR3) {};

  \node[leafopen]   at (LLLL4) {};
  \node[leafopen]   at (LLLR4) {};
  \node[leafopen]   at (LLRL4) {};
  \node[leafopen]   at (LLRR4) {};

  \node[leafopen]   at (LRLL4) {};
  \node[leafopen]   at (LRLR4) {};
  \node[leafopen]   at (LRRL4) {};
  \node[leaffilled] at (LRRR4) {}; 

  \node[leafopen]   at (RLLL4) {};
  \node[leafopen]   at (RLLR4) {};
  \node[leafopen]   at (RLRL4) {};
  \node[leaffilled] at (RLRR4) {}; 

  \node[leafopen]   at (RRLL4) {};
  \node[leafopen]   at (RRLR4) {};
  \node[leafopen]   at (RRRL4) {};
  \node[leafopen]   at (RRRR4) {};

  \def\arrowshift{0.6} 

  \coordinate (sLLLL4) at ($(LLLL4)+(0,-\arrowshift)$);
  \coordinate (sLLLR4) at ($(LLLR4)+(0,-\arrowshift)$);
  \coordinate (sLLRL4) at ($(LLRL4)+(0,-\arrowshift)$);
  \coordinate (sLLRR4) at ($(LLRR4)+(0,-\arrowshift)$);

  \coordinate (sLRLL4) at ($(LRLL4)+(0,-\arrowshift)$);
  \coordinate (sLRLR4) at ($(LRLR4)+(0,-\arrowshift)$);
  \coordinate (sLRRL4) at ($(LRRL4)+(0,-\arrowshift)$);
  \coordinate (sLRRR4) at ($(LRRR4)+(0,-\arrowshift)$);

  \coordinate (sRLLL4) at ($(RLLL4)+(0,-\arrowshift)$);
  \coordinate (sRLLR4) at ($(RLLR4)+(0,-\arrowshift)$);
  \coordinate (sRLRL4) at ($(RLRL4)+(0,-\arrowshift)$);
  \coordinate (sRLRR4) at ($(RLRR4)+(0,-\arrowshift)$);

  \coordinate (sRRLL4) at ($(RRLL4)+(0,-\arrowshift)$);
  \coordinate (sRRLR4) at ($(RRLR4)+(0,-\arrowshift)$);
  \coordinate (sRRRL4) at ($(RRRL4)+(0,-\arrowshift)$);
  \coordinate (sRRRR4) at ($(RRRR4)+(0,-\arrowshift)$);

  \coordinate (A1) at (-3.60, -5.6-\arrowshift);
  \coordinate (A2) at (-1.68, -5.6-\arrowshift);
  \coordinate (A3) at (-0.48, -5.6-\arrowshift);
  \coordinate (A3') at (-0.97, -5.6-\arrowshift);
  \coordinate (A4) at ( 0.72, -5.6-\arrowshift);
  \coordinate (A5) at ( 1.92, -5.6-\arrowshift);
  \coordinate (A6) at ( 3.12, -5.6-\arrowshift);
    \coordinate (A6') at ( 4.08, -5.6-\arrowshift);

  \draw[->] (sLLLL4) -- (A1);
  \draw[->] (sLLLR4) -- (A1);
  \draw[->] (sLLRL4) -- (A1);
  \draw[->] (sLLRR4) -- (A1);
  \node[below=1pt] at (A1) {$A_1$};

  \draw[->] (sLRLL4) -- (A2);
  \draw[->] (sLRLR4) -- (A2);
  \node[below=1pt] at (A2) {$A_2$};

  \draw[->] (sLRRL4) -- (A3');
  \draw[->] (sLRRR4) -- (A3);
  \node[below=1pt] at (A3) {$A_4$};
  \node[below=1pt] at (A3') {$A_3$};

  \draw[->] (sRLLL4) -- (A4);
  \draw[->] (sRLLR4) -- (A4);
  \node[below=1pt] at (A4) {$A_5$};

  \draw[->] (sRLRL4) -- (A4);
  \draw[->] (sRLRR4) -- (A5);
  \node[below=1pt] at (A5) {$A_6$};

  \draw[->] (sRRLL4) -- (A6);
  \draw[->] (sRRLR4) -- (A6);
  \draw[->] (sRRRL4) -- (A6');
  \draw[->] (sRRRR4) -- (A6');
  \node[below=1pt] at (A6) {$A_7$};
  \node[below=1pt] at (A6') {$A_8$};

  \node[anchor=west] at (-6.24,  0.8) {Depth};
  \node[anchor=west] at (-6.24,  0.0) {$0$};
  \node[anchor=west] at (-6.24, -1.2) {$1$};
  \node[anchor=west] at (-6.24, -2.4) {$h=2$};
  \node[anchor=west] at (-6.24, -3.6) {$3$};
  \node[anchor=west] at (-6.24, -4.8) {$4$};
\end{tikzpicture}

    \caption{Illustrating the construction of the admissible sequence. Assume that the Galton--Watson tree is binary. Suppose that $u>0$ is some threshold, $h=2$, and we use solid dots to represent vertices $v$ such that $|m^{-H\ell(v)}\eta_v|> u$, where we recall that $\ell(v)$ is the depth of $v$. Assume also that all such vertices have depth at most 4. Our construction ensures that two rays belong to the same set in the partition only if they have a common ancestor at depth $h$ and the deepest solid dots on them coincide. }
    \label{fig:1}
\end{figure}

\subsection{Removing bad events}

Before controlling the diameter $\Delta(A_n(t))$, we need to remove certain irregular behavior of the random set $S\subseteq \ell^2$, which is the goal of the current section.

Let $K>0$ be a large constant to be determined, and pick $q'\in(1/q,H)$. Recall from \eqref{eq:undef} that the sequence $\{u_j\}_{j\geq 1}$ is eventually decreasing. To control the differences of $\n{s-t}_2$ for $s,t\in A\in\A_n$, it remains to control the number of coordinates within each $s,t$ that have a size larger than $u_j$, for each $j\geq n$. Such size is controlled using the event $E_K$ defined as 
\begin{align}
   \begin{split}
        E_K:=\Big\{\forall &n\geq \max\{N_1,N_2,N_3\},~\forall j\geq n,~\forall v\in \C_n,\\
    &\max_{t\in\T_v}\#\{w\in t:\,|m^{-H\ell(w)}\eta_w|> u_j\}\leq 2+(2^j-j-2^{n-2})(\log_m2)+3KP2^{q'2^j}\Big\}.
   \end{split}\label{eq:Ek def}
\end{align}
Note that, by definition, for $t\in\T_v$, $\#\{w\in t:\,|m^{-H\ell(w)}\eta_w|> u_j\}$ counts vertices below (but not above) $v$, since each ray $t$ starts from the root $v$ of $\T_v$.

We apply a union bound to control the size of its complement $E^c_K$:
\begin{align*}
    &\p(E^c_K)\\&\leq \sum_{n\geq \max\{N_1,N_2,N_3\}}\sum_{j\geq n}\E\bigg[\sum_{v\in \C_n}\bone_{\{\max_{t\in\T_v}\#\{w\in t:\,|m^{-H\ell(w)}\eta_w|> u_j\}> 2+(2^j-j-2^{n-2})(\log_m2)+3KP2^{q'2^j}\}}\bigg]\\
    &\leq \sum_{n\geq \max\{N_1,N_2,N_3\}}\sum_{j\geq n}\sum_{k\geq \lfloor2^{n-2}(\log_m2)\rfloor}\\
    &\hspace{0.5cm}\E\bigg[\sum_{v\in V_k}\bone_{\{v\in \C_n\}}\bone_{\{\max_{t\in\T_v}\#\{w\in t:\,|m^{-H\ell(w)}\eta_w|> u_j\}> 2+(2^j-j-2^{n-2})(\log_m2)+3KP2^{q'2^j}\}}\bigg].
\end{align*}
Recall from the construction in Section \ref{sec:An} that if $n\geq \max\{N_1,N_2,N_3\}$ and $v\in\C_n$, there are two (not necessarily disjoint) possibilities:
\begin{enumerate}[(i)]
   
    \item $v\in V_k$ for some $k\geq \lfloor2^{n-2}(\log_m2)\rfloor$ and $|m^{-H\ell(v)}\eta_v|>u_{n-1}$;
     \item $v\in V_{\lfloor2^{n-2}(\log_m2)\rfloor}$.
\end{enumerate}

In case (i), the contribution is at most 
\begin{align*}
    &\sum_{n\geq \max\{N_1,N_2,N_3\}}\sum_{j\geq n}\sum_{k\geq \lfloor2^{n-2}(\log_m2)\rfloor}\\
    &\hspace{0.5cm}\E\bigg[\sum_{v\in V_k}\bone_{\{|m^{-kH}\eta_v|> u_{n-1}\}}\bone_{\{\max_{t\in\T_v}\#\{w\in t:\,|m^{-H\ell(w)}\eta_w|> u_j\}> 2+(2^j-j-2^{n-2})(\log_m2)+3KP2^{q'2^j}\}}\bigg].
\end{align*}
Conditional on the tree up to depth $k$, the subtrees below distinct depth-$k$ vertices are i.i.d.~and independent of the tree up to depth $k$. Therefore, by independence, 
\begin{align}
   \begin{split}
        & \sum_{n\geq \max\{N_1,N_2,N_3\}}\sum_{j\geq n}\sum_{k\geq \lfloor2^{n-2}(\log_m2)\rfloor}\\
    &\hspace{0.5cm}\E\bigg[\sum_{v\in V_k}\bone_{\{|m^{-kH}\eta_v|> u_{n-1}\}}\bone_{\{\max_{t\in\T_v}\#\{w\in t:\,|m^{-H\ell(w)}\eta_w|> u_j\}> 2+(2^j-j-2^{n-2})(\log_m2)+3KP2^{q'2^j}\}}\bigg]\\
    &\leq \sum_{n\geq \max\{N_1,N_2,N_3\}}\sum_{j\geq n}\sum_{k\geq \lfloor2^{n-2}(\log_m2)\rfloor}m^k\p(|m^{-kH}Y|> u_{n-1})\\
    &\hspace{0.5cm}\times\p\Big(\max_{t\in\T_v}\#\{w\in t:\,|m^{-H\ell(w)}\eta_w|> u_j\}> 2+(2^j-j-2^{n-2})(\log_m2)+3KP2^{q'2^j}\Big),
   \end{split}\label{eq:EKc}
\end{align}
where in the last term we have $\ell(v)=k\geq \lfloor2^{n-2}(\log_m2)\rfloor$. 
 To evaluate the last probability, let $k'\geq k$ be the smallest integer such that $m^{Hk'}u_j\geq 1$. It follows from \eqref{eq:undef} that
\begin{align}
    k'-k&\leq k'+1-2^{n-2}(\log_m2)\leq 2+(2^j-j-2^{n-2})(\log_m2).\label{eq:kk'}\end{align}
We decompose the rays in the tree $\T_v$ further into subtrees at depth $k'$ (referenced to the original tree $\T$). There are $m^{k'-k}$ such subtrees in expectation, so by a union bound (applied after conditioning on the Galton--Watson tree up to depth $k'$) and a rescaling argument in the first step, and using \eqref{eq:kk'} in the second step, for $v\in V_k$,
\begin{align}
    &\p\Big(\max_{t\in\T_v}\#\{w\in t:\,|m^{-H\ell(w)}\eta_w|> u_j\}> 2+(2^j-j-2^{n-2})(\log_m2)+3KP2^{q'2^j}\Big)\nonumber\\
    &\leq m^{k'-k}\p\Big(\max_{t\in\T}\#\{w\in t:\,|m^{-H\ell(w)}\eta_w|>1\}\geq 3KP2^{q'2^j}\Big)\label{eq:tree rescaling}\\
    &\leq m^22^{2^j-j-2^{n-2}} \p\Big(\max_{t\in\T}\#\{w\in t:\,|m^{-H\ell(w)}\eta_w|>1\}\geq 3KP2^{q'2^j}\Big)\nonumber.
\end{align}
By Lemma \ref{lemma:tree exceedance} applied with $r=3KP2^{q'2^j}$ and $\alpha=K2^{q'2^j}$ (here we also use $P>0$), we have uniformly for $K$ large enough (possibly depending on $P,q,H,q'$),
\begin{align*}
    \p\Big(\max_{t\in\T}\#\{w\in t:\,|m^{-H\ell(w)}\eta_w|>1\}\geq 3KP2^{q'2^j}\Big)&\leq \exp(-3KP2^{q'2^j}(\log3-1))+CK^{-q}2^{-qq'2^j}\\
    &\leq (C+1)K^{-q}2^{-qq'2^j}.
\end{align*}
Inserting in \eqref{eq:tree rescaling}, we have uniformly for $K$ large enough and $v\in V_k$ with $k\geq \lfloor2^{n-2}(\log_m2)\rfloor$,
\begin{align}
    \begin{split}
        &\p\Big(\max_{t\in\T_v}\#\{w\in t:\,|m^{-H\ell(w)}\eta_w|> u_j\}> 2+(2^j-j-2^{n-2})(\log_m2)+3KP2^{q'2^j}\Big)\\
    &\leq (C+1)K^{-q}m^22^{2^j-j-2^{n-2}-qq'2^j}.\label{eq:p}
    \end{split}
\end{align}
Combining \eqref{eq:EKc} and \eqref{eq:p} yields a total contribution of at most
\begin{align}
    \begin{split}
   &     \sum_{n\geq \max\{N_1,N_2,N_3\}}\sum_{j\geq n}\sum_{k\geq \lfloor2^{n-2}(\log_m2)\rfloor}m^k\p(|m^{-kH}Y|> u_{n-1})(C+1)K^{-q}m^22^{2^j-j-2^{n-2}-qq'2^j}\\
    &\leq C_{m,P}\sum_{n\geq \max\{N_1,N_2,N_3\}}\sum_{j\geq n}u_{n-1}^{-1/H}(C+1)K^{-q}m^22^{2^j-j-2^{n-2}-qq'2^j},\end{split}\label{eq:casei}
\end{align}
where we have used Lemma \ref{lem:polyLambda}. 

In case (ii), the total contribution is at most
\begin{align*}
   \begin{split}
        &\sum_{n\geq \max\{N_1,N_2,N_3\}}\sum_{j\geq n}\\
        &\hspace{1cm}\E\bigg[\sum_{v\in V_{\lfloor2^{n-2}(\log_m2)\rfloor}}\bone_{\{\max_{t\in\T_v}\#\{w\in t:\,|m^{-H\ell(w)}\eta_w|> u_j\}> 2+(2^j-j-2^{n-2})(\log_m2)+3KP2^{q'2^j}\}}\bigg]\\
    &\leq \sum_{n\geq \max\{N_1,N_2,N_3\}}\sum_{j\geq n}\\
        &\hspace{1cm}2^{2^{n-2}}\p\Big(\max_{t\in\T_v}\#\{w\in t:\,|m^{-H\ell(w)}\eta_w|> u_j\}> 2+(2^j-j-2^{n-2})(\log_m2)+3KP2^{q'2^j}\Big),
   \end{split}
\end{align*}
where $v\in V_{\lfloor2^{n-2}(\log_m2)\rfloor}$ in the last probability. By \eqref{eq:p}, this is at most
\begin{align}
    \sum_{n\geq \max\{N_1,N_2,N_3\}}\sum_{j\geq n}2^{2^{n-2}}(C+1)K^{-q}m^22^{2^j-j-2^{n-2}-qq'2^j}.\label{eq:caseii}
\end{align}

Combining \eqref{eq:casei} and \eqref{eq:caseii} and using \eqref{eq:undef}, we have
\begin{align*}
    \p(E^c_K)&\leq (C_{m,P}+1) \sum_{n\geq \max\{N_1,N_2,N_3\}}\sum_{j\geq n}2^{2^{n-2}}(C+1)K^{-q}m^22^{2^j-j-2^{n-2}-qq'2^j}.
\end{align*}
Recall that $qq'>1$. 
Therefore, uniformly for $K$ large enough,
\begin{align}
  \begin{split}
        \p(E^c_K)    &\leq (C_{m,P}+1)(C+1)K^{-q}m^2\sum_{n\geq \max\{N_1,N_2,N_3\}}\sum_{j\geq n}2^{(1-qq')2^j-j}\\
    &\leq C_{q,q'}(C_{m,P}+1)(C+1)K^{-q}m^2\sum_{n\geq \max\{N_1,N_2,N_3\}}2^{(1-qq')2^n-n}\\
    &\leq C'K^{-q},
  \end{split}\label{eq:prob estimate}
\end{align}
for some constant $C'>0$ possibly depending on $m,P,q,q',H$ but not on $K$. 

\subsection{Proof of Theorem \ref{thm}}

It remains to prove that, for $K$ large enough, we have $\gamma_2(S_1(\T))<\infty$ on the event $E_K$. This finishes the proof of the boundedness as we send $K\to\infty$ in \eqref{eq:prob estimate}.

Let $N_4\in\N$ be such that $u_n$ defined in \eqref{eq:undef} is decreasing for $n\geq N_4$. 
Let the partition $\{\A_n\}_{n\geq 0}$ be as constructed from Section \ref{sec:An} and consider $n>\max_{1\leq j\leq 4}N_j$.  Let $A\in\A_n$ and $s,t\in A$ be arbitrary. By definition, $s,t$ have a common ancestor $v\in\C_n$.   
By the definition \eqref{eq:Ek def}, on the event $E_K$ with $K$ chosen large enough (depending on $m,P,q'$ but not on $n,j$), for each $j\geq n$, 
\begin{align}
    \begin{split}
        \#\big\{w\in t:\,|m^{-H\ell(w)}\eta_w|> u_j\big\}&\leq 2+(2^j-j-2^{n-2})(\log_m2)+3KP2^{q'2^j}\leq 4KP2^{q'2^j},
    \end{split}\label{eq:EK bound}
\end{align}
and the same holds with $t$ replaced by $s$. To control $\n{s-t}_2$, it remains to control it on the uncommon part of $s,t$ in the subtree with root $v$ and use $\n{s-t}_2\leq\n{s}_2+\n{t}_2$ on this part. Using \eqref{eq:EK bound} and Minkowski's inequality, we thus arrive at the upper bound 
\begin{align}
\begin{split}
     \n{s-t}_2 &\leq \sum_{j\geq n}u_{j-1}\Big(\sqrt{\#\{w\in t:\ell(w)>\ell(v),\,u_j<|m^{-H\ell(w)}\eta_w|\leq  u_{j-1}\}}\\
     &\hspace{2cm}+\sqrt{\#\{w\in s:\ell(w)>\ell(v),\,u_j<|m^{-H\ell(w)}\eta_w|\leq u_{j-1}\}}\Big)\\
     &\leq \sum_{j\geq n}u_{j-1}\Big(\sqrt{\#\{w\in t:\ell(w)>\ell(v),\,|m^{-H\ell(w)}\eta_w|> u_j\}}\\
    &\hspace{2cm}+\sqrt{\#\{w\in s:\ell(w)>\ell(v),\,|m^{-H\ell(w)}\eta_w|> u_j\}}\Big)\\
    &\leq 2\sum_{j\geq n}u_{j-1}\sqrt{4KP2^{q'2^j}}\\
    &\leq 4C_{q',H}\sqrt{KP}\, 2^{(q'-H)2^{n-1}}2^{H(n-1)}
\end{split}\label{eq:st}
\end{align}
for some constant $C_{q',H}>0$ that depends only on $q',H$. 
Since it holds uniformly for $A\in\A_n$ and $s,t\in A$, we have the uniform bound $\Delta(A_n(t))\leq 4C_{q',H}\sqrt{KP}\, 2^{(q'-H)2^{n-1}}2^{H(n-1)}$ for $t\in S_1$ and $n> \max_{1\leq j\leq 4}N_j$.

In addition, by the triangle inequality, with $n=\max_{1\leq j\leq 4}N_j+1$ and choices $s_A\in A$, 
\begin{align}
    \Delta(S_1)\leq 2\max_{A\in\A_n}\Delta(A)+\max_{A,A'\in\A_n}\n{s_A-s_{A'}}_2<\infty.\label{eq:finite diameter}
\end{align}
Note that this implies our earlier claim that $S\subseteq\ell^2$, since a single ray $t\in S$ satisfies $t\in \ell^2$ a.s.\footnote{This follows from exactly the same arguments leading to \eqref{eq:st}, with $\n{s-t}_2$ therein replaced by $\n{t}_2$.}~and $\Delta(S_2)$ is clear from $P<\infty$. 
Moreover, it follows that
\begin{align}
    \begin{split}
        \gamma_2(S_1(\T))&\leq \sum_{n=0}^{ \max_{1\leq j\leq 4}N_j}2^{n/2}\Delta(S_1)+ \sup_{t\in S_1(\T)}\sum_{n> \max_{1\leq j\leq 4}N_j}2^{n/2}4C_{q',H}\sqrt{KP}\, 2^{(q'-H)2^{n-1}}2^{H(n-1)}\\
    &\leq 2^{\frac{\max_{1\leq j\leq 4}N_j}{2}+2}\Delta(S_1)+4C_{q',H}\sqrt{KP} \sum_{n> \max_{1\leq j\leq 4}N_j}2^{n/2}2^{(q'-H)2^{n-1}}2^{H(n-1)}<\infty.
    \end{split}\label{eq:final gamma}
\end{align}
This proves \eqref{eq:t} and completes the proof of Theorem \ref{thm}.

\begin{remark}
    In our proof, we have used multiple times the fact that the discounting factor $m^{-iH}$ in \eqref{eq:discounted BRW} is of geometric form. For example, this is needed in \eqref{eq:tree rescaling} when we rescale the branching random walk genealogy. However, we expect that our arguments can be extended to more general scalings that are close to geometric. 
\end{remark}

\subsection{Finiteness of the expected supremum}

With a bit extra effort, one can prove that $\E[ \gamma_2(S_1(\T))]<\infty$ and if further $H\in(0,1)$, we have $\E[b(S(\T))]<\infty$. This will be the goal of this section, and the techniques will also be used in proving Theorem \ref{thm:final} in the general asymmetric case in Section \ref{sec:general}.

We start from tail estimates of $N_1,\,N_3,$ and $\Delta(S_1(\T))$. To this end, we recall the definitions of $N_1$ and $N_3$ from Section \ref{sec:An}: $N_1$ is the smallest integer satisfying
\begin{align}
    \#V_k\leq k^2m^k\text{ a.s. for }k\geq h(N_1)\text{ where }h(n)=\lfloor 2^{n-2}(\log_m2)\rfloor,\label{eq:N1}
\end{align}
and $N_3$ is the smallest integer satisfying
$$\#\{v\in V:|m^{-H\ell(v)}\eta_v|> u_n\}\leq 2^{2^n}-1\text{ for }n\geq N_3.$$
In the following, the probability and expectation operators are taken with respect to the randomness from the discounted branching random walk.


\begin{lemma}\label{lemma:1}
    It holds that for some $C_1>0$, $\p(N_1>n)\leq C_1 2^{-n}$.
\end{lemma}

\begin{proof}
    By the definition \eqref{eq:N1}, a union bound, and Markov's inequality, we have 
    \begin{align*}
        \p(N_1>n)&=\p(\exists k\geq h(n),~W_k>k^2)\\
        &\leq \sum_{k\geq \lfloor 2^{n-2}(\log_m2)\rfloor}\p\Big(\sup_{j\geq 1}W_j>k^2\Big)\\
        &\leq \E\Big[\Big(\sup_{j\geq 1}W_j\Big)^q\Big]\sum_{k\geq \lfloor 2^{n-2}(\log_m2)\rfloor}k^{-2q}\\
        &\leq C_12^{-n}
    \end{align*}
 for some $C_1>0$, where we have used $q>1$ and the fact that $\E[(\sup_{j\geq 1}W_j)^q]<\infty$ from the proof of Lemma \ref{lemma:tree exceedance}.
\end{proof}

\begin{lemma}\label{lemma:2}
    It holds that for some $C_2>0$, $\p(N_3>n)\leq C_22^{-n}$.
\end{lemma}

\begin{proof}
    This follows directly from Markov's inequality applied to \eqref{eq:M} and a union bound:
    $$\p(N_3>n)\leq \sum_{k\geq n}\p(\#\{v\in V:|m^{-H\ell(v)}\eta_v|> u_k\}> 2^{2^k}-1)\leq \sum_{k\geq n}C_{m,P}2^{1-k}\leq C_22^{-n}$$
    for some $C_2>0$.
\end{proof}

\begin{lemma}\label{lemma:diameter bound}
    There exists $C_3,C_4>0$ such that for all $K$ large enough,
    $$\p(\Delta(S_1(\T))> C_3\sqrt{K})\leq C_4K^{-q}.$$
\end{lemma}
\begin{proof}
    Define the event 
    $$\cE_K:=\Big\{\forall j\geq 1,~\max_{t\in\T}\#\{w\in t:\,|m^{-H\ell(w)}\eta_w|> u_j\}\leq 2+(2^j-j)(\log_m2)+3KP2^{q'2^j}\Big\}.$$
    Using the same arguments that lead to \eqref{eq:prob estimate}, there exists $C_4>0$ such that $\p(\cE_K^c)\leq C_4K^{-q}$ for $K$ large enough. On the event $\cE_K$, we have by a similar computation to \eqref{eq:st} that $\Delta(S_1(\T))\leq C_3\sqrt{K}$ for some $C_3>0$.
\end{proof}

\begin{proposition}\label{prop:finite E}
    In the same setting as Theorem \ref{thm}, we have $\E[\gamma_2(S_1(\T))]<\infty$.\footnote{In other words, a discounted branching random walk restricted to increments in $[-1,1]$ has a finite expected supremum, regardless of whether the increment $Y$ has a finite first moment or not.} If $H\in(0,1)$, we further have $\E[\sup_{t\in\T}X_t]<\infty$. 
\end{proposition}

\begin{proof}
Following \eqref{eq:final gamma}, we have that modulo an event $E_K^c$ with probability $\p(E_K^c)\leq C'K^{-q}$,
$$\gamma_2(S_1(\T))\leq C_1'\big(2^{N_1/2}+2^{N_3/2}\big)\Delta(S_1(\T))+C_2'\sqrt{K}$$
for some $C_1',C_2'>0$ independent of $K$. By Lemma \ref{lemma:diameter bound}, there exist $C_3',C_4'>0$ independent of $K$ such that for $K$ large enough,
$$\p\Big(\gamma_2(S_1(\T))\geq C_3'\big(2^{N_1/2}+2^{N_3/2}\big)\sqrt{K}\Big)\leq C_4'K^{-q}.$$
     Since $q>1$, a standard tail integration argument and Lemmas \ref{lemma:1} and \ref{lemma:2} altogether yield that $\E[\gamma_2(S_1(\T))]<C_5'\E[2^{N_1/2}+2^{N_3/2}]<C_6'$ for some $C_5',C_6'<\infty$. 
     
     To deal with increments outside $[-1,1]$ (the set $S_2(\T)$), note the deterministic inequality
     $$\sup_{t\in\T}\n{s_2(t)}_1\leq \sum_{v\in V}|m^{-H\ell(v)}\eta_v|\bone_{\{|m^{-H\ell(v)}\eta_v|>1\}}.$$
     If $\E[|Y|^{1/H}]<\infty$ for some $H\in(0,1)$, 
\begin{align*}
   \E\Big[\sup_{t\in\T}\n{s_2(t)}_1\Big]&\leq \sum_{k\geq 1}m^k\E\big[|m^{-kH}Y|\bone_{\{|m^{-kH}Y|>1\}}\big]\\
   &=\int_0^\infty\p(|Y|>t)\sum_{k:~m^{kH}\leq t}m^{k(1-H)}\d t\\
   &\leq O(1)+C_{m,H}\int_1^\infty \p(|Y|>t)\,t^{1/H-1}\d t\\
   &\leq O(1)+HC_{m,H}\E[|Y|^{1/H}].
\end{align*}
   Therefore, $\E[\sup_{t\in\T}\n{s_2(t)}_1]<C_7'$ for some $C_7'<\infty$. We thus conclude from Lemma \ref{lemma:Bernoulli} that if $H\in(0,1)$, $b(S(\T))<\infty$, i.e., the supremum $\sup_{t\in\T}X_t$ has a finite expectation. 
\end{proof}

\section{Proof of the general asymmetric case}\label{sec:general}

\subsection{Sufficiency for boundedness in the positive case}

In this section, we relax the symmetry assumption of Theorem \ref{thm}, focusing on the setting where $Y$ is positive. As we discussed in the introduction, this case is of particular interest for the study of functional integral equations.

\begin{theorem}\label{thm:asym}
    Suppose that $m>1$ and $Y\geq 0$. If there exists $q>\max\{1/H,1\}$ such that $\E[Z^q]<\infty$ and $\E[Y^{1/H}]<\infty$, then $\sup_{t\in\T}X_t<\infty$ a.s.
\end{theorem}



\begin{proof}
We split into two cases: $H\in(0,1]$ and $H\in(1,\infty)$. Consider first $H\in(1,\infty)$. Using the same proof that leads to \eqref{eq:s2}, there are a.s.~finitely many $v\in V$ with $m^{-H\ell(v)}\eta_v>1$, so it remains to show that the sum of the values $m^{-H\ell(v)}\eta_v$ over all vertices $v$ with $m^{-H\ell(v)}\eta_v\leq 1$ has a finite expectation. By Lemma \ref{lem:polyLambda}, we have
\begin{align*}
    \sum_{k\geq 1}m^k\E\big[m^{-kH}Y\bone_{\{Y\leq m^{kH}\}}\big]&=\int_0^1\sum_{k\geq 1}m^k\p(Y>um^{kH})\,\d u\\
    &\leq C_{m,P}\int_0^1u^{-1/H}\d u<\infty.
\end{align*}
This proves the desired boundedness for $H\in(1,\infty)$.

Suppose now that $H\in(0,1]$, so by our assumption, $\E[Y]<\infty$. We write the decomposition
\begin{align*}
    X_t=\sum_{i\geq 1} m^{-iH}\eta_{t_i}&=\sum_{i\geq 1} m^{-iH}\eta_{t_i}\bone_{\{m^{-iH}\eta_{t_i}\leq  1\}}+\sum_{i\geq 1} m^{-iH}\eta_{t_i}\bone_{\{m^{-iH}\eta_{t_i}> 1\}}\\
    &=:X^{(1)}_t+X^{(2)}_t,\quad t\in\T.
\end{align*}
Note that this is the same decomposition as \eqref{eq:s1s2}. Using the same argument that leads to \eqref{eq:s2}, we have that $\{X^{(2)}_t\}_{t\in\T}$ is bounded as it almost surely contains finitely many non-zero terms. Therefore, it remains to prove $\sup_{t\in\T}X^{(1)}_t<\infty$ a.s. To this end, we use an alternative representation of $\{X^{(1)}_t\}_{t\in\T}$. Let $\{\zeta_v\}_{v\in V}$ be a collection of independent random variables such that for each Borel set $B\subseteq\R$, $$\p(\zeta_v\in B)=\p(m^{-iH}Y\in B\cap[0,1])+\p(Y>m^{iH})\chi_B(0),\quad v\in V_i,\quad i\in\N,$$
where $\chi_B(0)=1$ if $0\in B$ and $\chi_B(0)=0$ otherwise. As a consequence, we have the representation
$$X^{(1)}_t=\sum_{i\geq 1}\zeta_{t_i},\quad t\in\T$$
and
\begin{align}
 \E[X^{(1)}_t]\leq \sum_{i\geq 1}\E[m^{-iH}Y]\leq \frac{\E[Y]}{m^H-1}<\infty.\label{eq:finite E}
\end{align}
Next, let $\{\zeta_v'\}_{v\in V}$ be an independent copy of $\{\zeta_v\}_{v\in V}$. Define
\begin{align}
    \hat{X}_t:=\sum_{i\geq 1}\big(\zeta_{t_i}-\zeta_{t_i}'\big),\quad t\in\T.\label{eq:hat X}
\end{align}
We now obtain a symmetric branching random walk in time-inhomogeneous environment, i.e., the increments are still independent, but their distributions depend on the depth $i$. This process is in general not a discounted branching random walk, but the behavior is similar, and we may apply essentially the same tools developed in this paper to study the behavior of the maximum, which is our next goal. 

Since the arguments are similar to a large extent, we briefly sketch the necessary changes instead of elaborating on the details. Consider the symmetric process \eqref{eq:hat X}. The arguments that control the random quantity $\gamma_2(S_1(\T))$ in the proof of Theorem \ref{thm} and its expectation in the proof of Proposition \ref{prop:finite E} do not directly require that the set $S_1(\T)$ is built from truncations of discounted branching random walks. Instead, they only require estimates on the total exceedance probabilities in Lemmas \ref{lem:polyLambda} and \ref{lemma:tree exceedance}. For notational simplicity, denote by $\{Y_i\}_{i\in\N}$ an independent sequence satisfying $\zeta_v\dd Y_{\ell(v)}$, and $\{Y_i'\}_{i\in\N}$ be another independent copy. For Lemma \ref{lem:polyLambda}, we now have for $u\in(0,2]$,
\begin{align*}
    \sum_{k\geq 1}m^k\p(|Y_k-Y_k'|>u)&\leq 2\sum_{k\geq 1}m^k\p\Big(|Y_k|>\frac{u}{2}\Big)\\
    &\leq 2\sum_{k\geq 1}m^k\p\Big(|Y|>\frac{um^{kH}}{2}\Big)\\
    &\leq 2C_{m,P}\Big(\frac{u}{2}\Big)^{-1/H}=2^{1+1/H}C_{m,P}\,u^{-1/H}.
\end{align*}
For Lemma \ref{lemma:tree exceedance}, the constant $P$ from \eqref{eq:P} will be replaced by $P'$, where
\begin{align*}
    \sum_{k\geq 1}m^k\p(|Y_k-Y_k'|>1)&\leq 2\sum_{k\geq 1}m^k\p\Big(|Y_k|>\frac{1}{2}\Big)\\
    &\leq 2\sum_{k\geq 1}m^k\p\Big(|Y|>\frac{m^{kH}}{2}\Big)=:P'<\infty.
\end{align*}
These changes will be absorbed into constant factors, but do not change the other parts of the proofs. Proposition \ref{prop:finite E} (with suitable adaptation) then yields the existence of $C_6'>0$ such that 
$$\E\bigg[\sup_{t\in\T}\sum_{i\geq 1}\big(\zeta_{t_i}-\zeta_{t_i}'\big)\bigg]\leq C_6'.$$
We need yet one more observation. If we trim the process at time $i=N$ for some finite $N\in\N$, the finite-depth maxima possess the same upper bound, since the $\ell_2$ norm estimates $\n{s-t}_2$ can only become smaller if restricted to the first $N$ generations. This leads to 
$$\E\bigg[\max_{t\in\T}\sum_{i=1}^N\big(\zeta_{t_i}-\zeta_{t_i}'\big)\bigg]\leq C_6',\quad N\in\N.$$
By Jensen's inequality and since the maximum of linear functions is convex, we have
$$\E\bigg[\max_{t\in\T}\sum_{i=1}^N\big(\zeta_{t_i}-\E[\zeta_{t_i}]\big)\bigg]\leq \E\bigg[\max_{t\in\T}\sum_{i=1}^N\big(\zeta_{t_i}-\zeta_{t_i}'\big)\bigg]\leq C_6',\quad N\in\N,$$
and therefore, by \eqref{eq:finite E},
$$\sup_{N\in\N}\E\bigg[\max_{t\in\T}\sum_{i=1}^N\zeta_{t_i}\bigg]\leq C_6'+\frac{\E[Y]}{m^H-1}.$$
Since each $\zeta_v\geq 0$, by the monotone convergence theorem, 
$$\E\Big[\sup_{t\in \T}X^{(1)}_t\Big]=\E\bigg[\sup_{t\in\T}\sum_{i\geq 1}{\zeta}_{t_i}\bigg]=\lim_{N\to\infty}\E\bigg[\max_{t\in\T}\sum_{i=1}^N{\zeta}_{t_i}\bigg]\leq C_6'+\frac{\E[Y]}{m^H-1}<\infty.$$
This completes the proof for $H\in(0,1]$.
\end{proof}

\subsection{Proofs of the main results}

\begin{proof}[Proof of Theorem \ref{thm:final}]
(i) It remains to show the following statement: with $Y_+:=\max\{0,Y\}$, if there exists $q>\max\{1/H,1\}$ such that $\E[Z^q]<\infty$ and $\E[Y_+^{1/H}]<\infty$, then $\sup_{t\in\T}X_t<\infty$.   If $\p(Y>0)=0$, we are done. Otherwise, let $\bar{Y}$ have the conditional law of $Y$ given $Y>0$, i.e., $\bar{Y}\dd Y\mid Y>0$. It is clear that ${Y}\lst \bar{Y}$, where $\lst$ refers to the first-order stochastic dominance. Let $\{\bar{X}_t\}_{t\in\T}$ denote the discounted branching random walk with the same offspring distribution $Z$ and discounting factor $H$ as those of $\{X_t\}_{t\in\T}$, but increments distributed according to $\bar{Y}$ instead of $Y$. A standard coupling argument shows that $\sup_{t\in\T}X_t\lst \sup_{t\in\T}\bar{X}_t$, so it remains to show that $\sup_{t\in\T}\bar{X}_t<\infty$ a.s. Since $\E[Y_+^{1/H}]<\infty$, we have $\E[\bar{Y}^{1/H}]=\E[Y_+^{1/H}]\,\p(Y>0)^{-1}<\infty.$ Theorem \ref{thm:asym} then directly implies $\sup_{t\in\T}\bar{X}_t<\infty$ a.s.

    (ii) First, we reduce to the symmetric case via symmetrization. Let $\{\eta_v'\}_{v\in V}$ be an independent copy of $\{\eta_v\}_{v\in V}$, and consider the process 
    $$ X^{\mathrm{sym}}_t:=\sum_{i\geq 1} m^{-iH}(\eta_{t_i}-\eta_{t_i}'),\quad t\in\T.$$
    This is a discounted branching random walk with increment distributed according to $Y-Y'$, where $Y'$ is an independent copy of $Y$. Note that $\E[|Y|^{1/H}]=\infty$ implies $\E[|Y-Y'|^{1/H}]=\infty$ (see Section 6.1 of \citep{LedouxTalagrand1991}). Suppose that we know the process $\{X^{\mathrm{sym}}_t\}_{t\in\T}$ is unbounded. Then by triangle inequality and using the fact that
    $$\sum_{i\geq 1} m^{-iH}\eta_{t_i}\dd \sum_{i\geq 1} m^{-iH}\eta_{t_i}',$$ the process $\{X_t\}_{t\in \T}$ must also be unbounded. Therefore, in the remainder of the proof, we assume that $Y$ has a symmetric distribution. 

    Assume that $\p(Z=0)=0$ and $\E[Z\log Z]<\infty$. Suppose for contradiction that $\sup_{t\in\T}X_t<\infty$ a.s. By the symmetrized representation \eqref{eq:dd}, we have 
\begin{align}
    \sup_{t\in\T}\sum_{i\geq 1}\ee_im^{-iH}\eta_{t_i}<\infty\quad\text{a.s.},\label{eq:finite}
\end{align}
where we recall that $\{\ee_i\}_{i\geq 1}$ are i.i.d.~Rademacher. We condition on $\{\eta_v\}_{v\in V}$ and our goal is to construct a $\sigma(\{\eta_v\}_{v\in V})$-measurable random variable $R\in\T$ such that the series 
$$\sum_{i\geq 1}\ee_im^{-iH}\eta_{R_i}$$
does not converge a.s., which contradicts \eqref{eq:finite}. 
 By Kolmogorov's three series theorem, a necessary condition for the series $\sum\ee_ix_i$ to be convergent is that $x_i\to 0$. Thus, it remains to construct $R$ such that $|m^{-iH}\eta_{R_i}|> 1$ infinitely often. 

By Lemma \ref{lem:polyLambda}, we have $\sum_{k\geq 1}m^k\p(|Y|>m^{kH})=\infty$. By the Kesten--Stigum theorem, $\#V_k/m^k\to W$ a.s.~for some strictly positive random variable $W$ (here we use $\p(Z=0)=0$), so $\sum_{k\geq 1}(\#V_k)\p(|Y|>m^{kH})=\infty$ a.s.  
 By the Borel--Cantelli lemma, there exists a.s.~$v\in V$ with $|m^{-H\ell(v)}\eta_v|>1$. Let the ray $R$ satisfy $R_{\ell(v)}=v$. Then, we proceed recursively to the subtree $\T_v$. Since $\sum_{k\geq \ell(v)+1}m^{k-\ell(v)}\p(|Y|>m^{kH})=\infty$, we find $w\succeq v$ with $|m^{-H\ell(w)}\eta_w|>1$ and let $R$ pass through $w$. Continuing this procedure, we find $R\in\T$ with infinitely many $i\geq 1$ such that $|m^{-iH}\eta_{R_i}|>1$. This finishes the proof.
\end{proof}


\begin{proof}[Proof of Proposition \ref{prop:finite E 2}]
    This follows from the same stochastic dominance argument as in the proof of Theorem \ref{thm:final}(i), using Proposition \ref{prop:finite E}.
\end{proof}

\section*{Acknowledgements}
 The author gratefully acknowledges support from a Jump Trading Fellowship, as well as valuable input from Yi Shen.

\bibliography{bib}
\bibliographystyle{plainnat}

\end{document}